\documentclass[12pt,reqno]{amsart}
\usepackage[a4paper,top=3cm,bottom=3cm,inner=3cm,outer=3cm]{geometry}

\usepackage{amsmath, amssymb, amsfonts, mathtools}
\usepackage{wasysym}

\usepackage{xcolor}
\definecolor{darkgreen}{rgb}{0,0.45,0}
\usepackage[
    colorlinks, 
    citecolor=blue, 
    urlcolor=darkgreen, 
    final, 
    hyperindex, 
    pagebackref, 
    linkcolor = black
]{hyperref}

\usepackage[capitalize]{cleveref}
\usepackage{xspace}

\usepackage{amsthm}

\theoremstyle{plain}
\newtheorem{thm}{Theorem}
\newtheorem*{thm*}{Theorem}
\newtheorem{prop}[thm]{Proposition}
\newtheorem{lem}[thm]{Lemma}
\newtheorem*{prop*}{Proposition}

\newtheorem{question}[thm]{Question}

\theoremstyle{definition}
\newtheorem{defn}[thm]{Definition}
\newtheorem{expl}[thm]{Example}

\theoremstyle{remark}

\title[Extensions of representation stable categories]
{Extensions of representation stable categories}

\author[Moeller]{Joe Moeller}

\DeclareFontFamily{U}{min}{}
\DeclareFontShape{U}{min}{m}{n}{<-> udmj30}{}

\newcommand{\define}[1]{{\bf \boldmath{#1}}\index{#1}}

\newcommand{\maps}{\colon}
\newcommand{\To}{\Rightarrow}
\newcommand{\inv}{^{-1}}
\newcommand{\inta}{\textstyle\int\hspace{-.07in}}
\newcommand{\Aut}{\mathsf{Aut}}

\newcommand{\CC}{\mathbb{C}}

\newcommand{\ZZ}{\mathbb Z}

\newcommand{\Hom}{\mathrm{Hom}}

\newcommand{\ob}{\mathrm{ob}}
\newcommand{\op}{^\mathrm{op}}

\newcommand{\category}[1]{\mathcal{#1}}
\newcommand{\A}{\category A}
\newcommand{\B}{\category B}
\newcommand{\C}{\category C}
\newcommand{\D}{\category D}

\newcommand{\X}{\category X}
\newcommand{\Y}{\category Y}

\newcommand{\pseudofunctor}[1]{\mathcal{#1}}

\newcommand{\M}{\pseudofunctor M}
\newcommand{\N}{\pseudofunctor N}

\newcommand{\namedcat}[1]{\mathsf{#1}}
\newcommand{\namedbicat}[1]{\mathsf{#1}}

\newcommand{\Cat}{\namedcat{Cat}}

\newcommand{\Cart}{\namedbicat{Cart}}

\newcommand{\Fib}{\namedcat{Fib}}
\newcommand{\FB}{\namedcat{FB}}
\newcommand{\FI}{\namedcat{FI}}

\newcommand{\Grp}{\namedcat{Grp}}
\newcommand{\FinGrp}{\namedcat{FinGrp}}

\newcommand{\ICat}{\namedcat{ICat}}

\newcommand{\Rep}{\namedcat{Rep}}

\newcommand{\Ring}{\namedcat{Ring}}

\newcommand{\Mod}{\namedcat{Mod}}

\mathchardef\mhyphen="2D

\usepackage{tikz, tikz-cd}
\usetikzlibrary{positioning}
\definecolor {processblue}{cmyk}{0.9,0.5,0,0}

\tikzstyle{simple}=[-,line width=2.000]
\tikzstyle{arrow}=[-,postaction={decorate},decoration={markings,mark=at position .5 with {\arrow{>}}},line width=1.100]
\pgfdeclarelayer{edgelayer}
\pgfdeclarelayer{nodelayer}
\pgfdeclarelayer{bg}
\pgfsetlayers{bg,edgelayer,nodelayer,main}
\tikzstyle{none}=[inner sep=-1pt]
\tikzstyle{empty}=[circle,fill=none, draw=none]
\tikzstyle{inputdot}=[circle,fill=black,draw=black, scale=.5]
\tikzstyle{dot}=[circle,fill=black,draw=black]
\tikzstyle{bounding}=[circle,dashed, fill=none,draw=black, scale=9.00]
\tikzstyle{simple}=[-,draw=black,line width=1.000]
\tikzstyle{inarrow}=[-,draw=black,postaction={decorate},decoration={markings,mark=at position .5 with {\arrow{>}}},line width=1.000]
\tikzstyle{tick}=[-,draw=black,postaction={decorate},decoration={markings,mark=at position .5 with {\draw (0,-0.1) -- (0,0.1);}},line width=1.000]
\tikzstyle{inputarrow}=[->,draw=black, shorten >=.05cm]

\tikzset{main node/.style={circle,fill=blue!20,draw,minimum size=1cm,inner sep=0pt},}

\tikzstyle{construct}=[fill=white, draw=black, shape=circle]
\tikzstyle{universal}=[fill=black, draw=black, shape=circle]

\numberwithin{thm}{section}

\begin{document}

\begin{abstract}
    A category of FI type is one which is sufficiently similar to finite sets and injections so as to admit nice representation stability results. Several common examples admit a Grothendieck fibration to finite sets and injections. We begin by carefully reviewing the theory of fibrations of categories with motivating examples relevant to algebra and representation theory. We classify which functors between FI type categories are fibrations, and thus obtain sufficient conditions for an FI type category to be the result of a Grothendieck construction.
\end{abstract}

\maketitle
\setcounter{tocdepth}{1} 
\tableofcontents

\section{Introduction}

One strategy when studying representations of a family of groups is to find a category where the automorphism groups are the groups of interest. A functor from this category into some some category of modules amounts to a family of representations of the groups equipped with some extra data. The seminal instance is that of the symmetric groups and the category of finite sets and injective functions, denoted $\FI$. Church--Ellenberg--Farb--Nagpal showed that over a noetherian ring $k$, any submodule of a finitely generated $\FI$-module is finitely generated \cite{FImodNoetherian}. The implication of this finite generation property for the family of $S_n$ representations underlying an $\FI$-module is that beyond a certain point, all the data of higher-indexed representations are completely determined by the lower ones. \emph{Quod est superius est sicut quod inferius.}

For each finite group $G$, there is a category $\FI_G$ which enjoys many of the same properties as $\FI$. Just as in $\FI$, the objects of $\FI_G$ are finite sets, but the morphisms are injections where each element of the source is decorated with an element of $G$. 
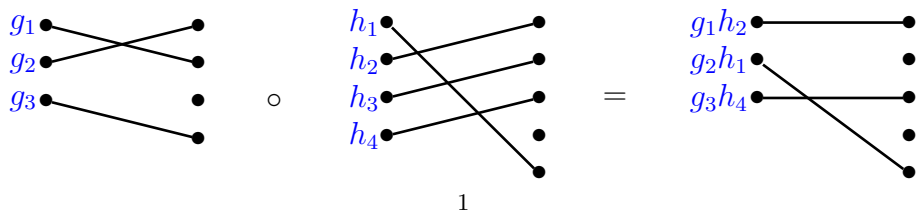
\[
\begin{tikzpicture}
	\begin{pgfonlayer}{nodelayer}
		\node [style=none]  (1) at (0, 0) {$\bullet$};
		\node [style=none]  (2) at (0, -0.5) {$\bullet$};
		\node [style=none]  (3) at (0, -1) {$\bullet$};
		\node [style=none]  (4) at (2, 0) {$\bullet$};
		\node [style=none]  (5) at (2, -0.5) {$\bullet$};
		\node [style=none]  (6) at (2, -1) {$\bullet$};
		\node [style=none]  (7) at (2, -1.5) {$\bullet$};
		\node [style=none] () at (3, -1) {$\circ$};
		\node [style=none]  () at (4, -2) {};
		\node [style=none]  () at (-0.3, 0) {\color{blue} $g_1$};
		\node [style=none]  () at (-0.3, -0.5) {\color{blue} $g_2$};
		\node [style=none]  () at (-0.3, -1) {\color{blue} $g_3$};
	\end{pgfonlayer}
	\begin{pgfonlayer}{edgelayer}
		\draw [style=simple] (1) to (5);
		\draw [style=simple] (2) to (4);
		\draw [style=simple] (3) to (7);
	\end{pgfonlayer}
\end{tikzpicture}
\begin{tikzpicture}
	\begin{pgfonlayer}{nodelayer}
		\node [style=none]  (1) at (0, 0) {$\bullet$};
		\node [style=none]  (2) at (0, -0.5) {$\bullet$};
		\node [style=none]  (3) at (0, -1) {$\bullet$};
		\node [style=none]  (4) at (0, -1.5) {$\bullet$};
		\node [style=none]  (5) at (2, 0) {$\bullet$};
		\node [style=none]  (6) at (2, -0.5) {$\bullet$};
		\node [style=none]  (7) at (2, -1) {$\bullet$};
		\node [style=none]  (8) at (2, -1.5) {$\bullet$};
		\node [style=none]  (9) at (2, -2) {$\bullet$};
		\node [style=none] () at (3, -1) {=};
		\node [style=none]  () at (4, -2) {};
		\node [style=none]  () at (-0.3, 0) {\color{blue} $h_1$};
		\node [style=none]  () at (-0.3, -0.5) {\color{blue} $h_2$};
		\node [style=none]  () at (-0.3, -1) {\color{blue} $h_3$};
		\node [style=none]  () at (-0.3, -1.5) {\color{blue} $h_4$};
	\end{pgfonlayer}
	\begin{pgfonlayer}{edgelayer}
		\draw [style=simple] (1) to (9);
		\draw [style=simple] (2) to (5);
		\draw [style=simple] (3) to (6);
		\draw [style=simple] (4) to (7);
	\end{pgfonlayer}
\end{tikzpicture}
\begin{tikzpicture}
	\begin{pgfonlayer}{nodelayer}
		\node [style=none]  (1) at (0, 0) {$\bullet$};
		\node [style=none]  (2) at (0, -0.5) {$\bullet$};
		\node [style=none]  (3) at (0, -1) {$\bullet$};
		\node [style=none]  (5) at (2, 0) {$\bullet$};
		\node [style=none]  (6) at (2, -0.5) {$\bullet$};
		\node [style=none]  (7) at (2, -1) {$\bullet$};
		\node [style=none]  (8) at (2, -1.5) {$\bullet$};
		\node [style=none]  (9) at (2, -2) {$\bullet$};
		\node [style=none]  () at (-0.5, 0) {\color{blue} $g_1h_2$};
		\node [style=none]  () at (-0.5, -0.5) {\color{blue} $g_2h_1$};
		\node [style=none]  () at (-0.5, -1) {\color{blue} $g_3h_4$};
	\end{pgfonlayer}
	\begin{pgfonlayer}{edgelayer}
		\draw [style=simple] (1) to (5);
		\draw [style=simple] (2) to (9);
		\draw [style=simple] (3) to (7);
	\end{pgfonlayer}
\end{tikzpicture}
\]
Composition is given by pulling back the group elements along the injection, and multiplying where necessary, as demonstrated below with $g_i, h_j \in G$. The automorphism groups here are precisely the wreath products $S_n \ltimes G^n$. Sam and Snowden defined this category and proved that its category of modules over a noetherian ring is noetherian \cite{RepGmaps}.

It can be seen without immense effort, once one is brought to ask the question, that the evident projection functor $\FI_G \to \FI$ is a \emph{fibration}. A \define{fibration of categories} is a functor $\A \to \X$ with a lifting property which allows one to formally pull back objects and arrows of $\A$ along arrows in $\X$. There is an equivalence between such fibrations and weak functors $\X\op \to \Cat$ valued in categories, functors, and natural transformations. The details of this classical equivalence can be found in \cite{2DCats}. Such a weak functor is known as an \define{$\X$-indexed category}. 
The reverse process, constructing a category with a fibration to $\X$ from an $\X$-indexed category, is called the \define{Grothendieck construction}.
The following vague, but natural, question arises. 
\begin{question}
\label{quest}
    If $\X$ is a category which is known to have nice representation theoretic properties, under which conditions does the Grothendieck construction of a weak functor $\X\op \to \Cat$ inherit those desired properties from $\X$?
\end{question}

The vagueness in  this question lies precisely in what is meant by ``nice representation theoretic properties''. Progress has been made to extract the essential categorical features of $\FI$ that allow for its apparent nice representation theoretic properties. Gan and Li \cite{EICat} have shown that under some simple combinatorial conditions, over a field of characteristic 0, finitely generated modules of an EI category (where every endomorphism is invertible) with objects essentially parameterized by $\N$ are noetherian. Gadish imposed some further categorical conditions, but was able to lift the demand for parameterization by $\N$, the motivating case being modules over $\FI^n$. In so doing, a noetherian result was obtained for a wider class of categories, designated \define{$\FI$-type categories} \cite{FItypeCats}. This paper seeks to provide an answer to \cref{quest} by providing necessary and sufficient conditions on an $\X$-indexed category such that the Grothendieck construction produces an $\FI$-type category, provided that $\X$ itself is of $\FI$ type.

Presently, we prove the following theorem in order to provide an answer to one instance of Question \ref{quest}. 
\begin{thm*}[\cref{thm:main}]
    Let $\X$ be an $\FI$-type category, and let $\M \maps \X\op \to \Cat$ be an indexed category. Then the Grothendieck construction applied to $\M$ produces an $\FI$-type category and the related fibration preserves pullbacks and weak pushouts if 
    \begin{enumerate}
        \item the fibers are $\FI$ type categories
        \item for every endomorphism $f \maps x \to x$ and object $a$ in the fiber over $x$, every map $a \to \M f(a)$ in $\M x$ is invertible
        \item the inclusions $\M x \hookrightarrow \inta \M$ preserve pullbacks
        \item $\M$ is weakly reversible
    \end{enumerate} 
    for all objects $x,y$ and morphisms $f \maps x \to y$ in $\X$. 
\end{thm*}
We do not expect the terms `weak pushout' or `weakly reversible' to be familiar to the reader, but are introduced within the paper where appropriate.

\cref{sec:Grothendieck} contains an overview of the theory of indexed categories, fibrations, and their equivalence via the Grothendieck construction. Examples relevant to representation theory motivate the definitions and constructions. 
\cref{sec:Auto} unpacks in detail what happens to automorphism groups under the construction, and the relation to nonabelian cohomology. 
In \cref{sec:GCFI}, we review Gadish's definition of $\FI$-type category and the representation stability results for such categories. We then state \cref{thm:main}, which gives sufficient conditions on an indexed category $\M \maps \X\op \to \Cat$ with an $\FI$-type base category, under which the Grothendieck construction produces an $\FI$-type total category, along with several examples. 
\cref{sec:proofs} is dedicated to the proof of \cref{thm:main}. The definition of FI type category consists of several categorical properties, and the relationship of each one with the notion of Grothendieck fibrations is treated separately. We hope this aids any future researcher in identifying which conditions they need for their own scenario.

\subsection*{Acknowledgements}
A tremendous debt of gratitude is owed to Derek Lowenberg, who introduced me to representation stability and helped me carve out the ideas presented here. Reid Barton and Mike Shulman pointed out instructive examples which helped to identify necessary and sufficient conditions for the Grothendieck construction to produce an EI category. I would also like to thank John Baez, Jonathan Beardsley, Spencer Breiner, Scott Carter, Nir Gadish, Wee Liang Gan, David Jaz Myers, Todd Trimble, and Christina Vasilakopoulou for helpful discussions.

\section{Grothendieck fibrations}
\label{sec:Grothendieck}

A \emph{Grothendieck fibration} is a functor $P \maps \A \to \X$ which essentially allows you to divide the category $\A$ into subcategories $\A_x$ called \emph{fibers} over the objects $x$ of $\X$, with \emph{pullback} functors between them corresponding to the maps in $\X$. Such structures naturally appear in fields as widely varying as algebraic geometry \cite{Vistoli}, logic and theoretical computer science \cite{Jacobs}, and homotopy theory and higher category theory \cite{FramedBicats}. They have been studied extensively by category theorists, especially in the context of topos theory \cite{Grayfibredandcofibred, FibredAdjunctions, Handbook2, Elephant1, alaBenabou, 2DCats}. 

The \emph{Grothendieck construction} produces such a fibration from a family of categories and functors indexed by the objects of the base category $\X$. In fact, any fibration is equivalent to one produced by this construction. In this way, there is an equivalence between fibrations and indexed categories. This offers two different perspectives of the same data, often allowing different tools and results to be used. 

In this section, we recall the basic theory of indexed categories, Grothendieck fibrations, their equivalence via the Grothendieck construction, and examples found in algebra and representation theory. The full generality of this theory demands the use of 2-categorical concepts and language. There are however 1-categorical versions of everything, which we do our best to include whenever possible. We encourage the unfamiliar reader to look to these first as they read through. 
Unabridged detail can be found in Johnson and Yau's book \cite{2DCats} for all 2-categorical concepts we use here, including 2-categories, pseudofunctors, pseudonatural transformations, and modifications. We will review an abbreviated form of this information for the reader's convenience.

\subsection{Indexed Categories}

We will use $\Cat$ to refer either to the category of categories and functors, or the 2-category of the same and natural transformations. We do our best to make the distinction clear either explicitly or by context. 

Let $\X$ be a category. A \define{strict $\X$-indexed category} is a functor $\M \maps \X\op \to \Cat$. An \define{$\X$-indexed functor} is a natural transformation $\alpha \maps \M \To \N$. Let $\ICat_s(\X)$ denote the functor category $[\X\op, \Cat]$.

\begin{expl}[Group representations]
\label{ex:Grep1}
    Let $\FinGrp$ denote the category of finite groups and group homomorphisms. Define the functor $\Rep \maps \FinGrp\op \to \Cat$ as follows. To a finite group $G$ assign the category of representations of $G$, $\Rep(G)$. To a group homomorphism $\phi \maps G \to H$ assign the functor $\phi^* \maps \Rep(H) \to \Rep(G)$ given by pulling back along $\phi$.
\end{expl}

\begin{expl}[Ring modules]
\label{ex:mod1}
    Let $\Ring$ denote the category of rings and ring homomorphisms. Define the functor $\Mod \maps \Ring\op \to \Cat$ as follows. To a ring $R$ it assigns the category of $R$-modules, $R\Mod$. To a ring homomorphism $\phi \maps R \to S$ assign the functor $\phi^\ast \maps S\Mod \to R\Mod$ given by pulling back along $\phi$, restriction of scalars.
\end{expl}

\begin{expl}[Group action as a functor]
\label{ex:action1}
    Groups can be thought of as one-object groupoids, and group homomorphisms as functors between these. Let $G$ and $H$ be groups and $A \maps G \to \Aut(H)$ a right action of $G$ on $H$. Denote the action of an element $g \in G$ on an element $h \in H$ by $h.g$. This can be thought of as a functor $A \maps G\op \to \Grp$ which sends the unique object of $G$ to $H$ itself, and every element of $G$ to the automorphisms specified by the action. By composing with the inclusion functor, we see this as an indexed category $G\op \xrightarrow A \Grp \hookrightarrow \Cat$. 
\end{expl}

Let $\C$ and $\D$ be 2-categories. A \define{pseudofunctor} $F \maps \C \to \D$ consists of an object function $\ob\C \to \ob\C$ just as functors do, but now the assignment on morphisms is a functor $\C(c, c') \to \D(Fc, Fc')$ (thus including a function on the 2-morphisms which is associative and unital). The composition and unit preservation laws at the level of 1-morphisms are now weakened from equations to specified invertible 2-morphisms $\mu_{f,g} \maps F(f) \circ F(g) \To F(f \circ g)$ natural in $f$ and $g$, and $\eta_c \maps id_{Fc} \To F(id_c)$ natural in $c$. We call these the \define{compositor} and \define{unitor} maps respectively. Further, these maps must themselves satisfy some equations, which can be found in Johnson and Yau's book \cite{2DCats}. 

There is a suitable generalization of natural transformations, called \define{pseudonatural transformations}. The naturality condition, instead of being the requirement for certain squares to commute, is extra structure in the form of invertible 2-morphisms, which again must satisfy some equations. Unlike with natural transformations, there is also a fitting notion of map \emph{between} pseudonatural transformations, called a \define{modification}. Just as a natural transformation assigns a morphism in the target category to each object of the source category, a modification assigns a 2-morphism in the target category to each object of the source category. Again, full detail can be found \emph{op cit}.

Let $\X$ be a category, thought of as a locally discrete 2-category. An \define{$\X$-indexed category} is a pseudofunctor $\M \maps \X\op \to \Cat$. An \define{$\X$-indexed functor} is a pseudonatural transformation $\alpha \maps \M \To \N$. A \define{$\X$-indexed natural transformation} is a modification. Let $\ICat(\X)$ denote the 2-category $[\X\op, \Cat]_{ps}$ of pseudofunctors, pseudonatural transformations, and modifications.

\begin{expl}[Slice and pullback]
\label{ex:slice}
    Let $\X$ be a category, and $x$ an object of $\X$. The \define{slice category at $x$}, denoted $\X/x$, has arrows $f \maps y \to x$ as objects and commutative triangles 
    \[
    \begin{tikzcd}
        y
        \arrow[rr, "h"]
        \arrow[dr, swap, "f"]
        &&
        z
        \arrow[dl, "g"]
        \\&
        x
    \end{tikzcd}
    \]
    as morphisms. Assume $\X$ has pullbacks. Given a map $j \maps x \to y$ in $\X$, then we define $j^* \maps \X/y \to \X/x$ on an object $f \maps z \to y$ by taking the following pullback.
    \[
    \begin{tikzcd}
        x \times_y z
        \arrow[r]
        \arrow[dr, phantom, "\lrcorner", pos = 0.1]
        \arrow[d, swap, "j^*f"]
        &
        z
        \arrow[d, "f"]
        \\
        x
        \arrow[r, swap, "j"]
        &
        y
    \end{tikzcd}
    \]
    For a morphism $h \maps f \to g$, $j^*(h)$ is provided by the universal property of pullbacks.
    \[
    \begin{tikzcd}
        &
        x \times_y w
        \arrow[ddrr, phantom, "\lrcorner", pos = 0.1]
        \arrow[rr]
        \arrow[dd, "j^*g", pos = 0.75]
        &&
        w
        \arrow[dd, "g"]
        \\
        x\times_y z
        \arrow[drrr, phantom, "\lrcorner", pos = 0.1]
        \arrow[ur, dashed, "j^*h"]
        \arrow[rr, crossing over]
        \arrow[dr, swap, "j^*f"]
        &&
        z
        \arrow[ur, "h"]
        \arrow[dr, swap, "f"]
        \\&
        x
        \arrow[rr, swap, "j"]
        &&
        y
    \end{tikzcd}
    \]
    Thus we define a pseudofunctor $\X/- \maps \X\op \to \Cat$ with the assignments on objects and morphisms as above, and the compositor and unitor derived from universal property of pullbacks. Despite these isomorphisms, this cannot in general be made into a strict functor \cite[Example 1.10.3.iv]{Jacobs, 279985}. 
\end{expl}

\subsection{The Grothendieck Construction}

Before fibrations, we discuss the Grothendieck construction. This construction builds a category, denoted $\inta\M$, out of the data of an indexed category $\M$. As we shall see, this category turns out to always be fibred, and all fibred categories arise this way.

Given an indexed category $\M \maps \X\op \to \Cat$, let $\inta \M$ denote the category with:
\begin{itemize}
    \item objects $(x,a)$ with $x \in \X$ and $a \in \M(x)$
    \item morphisms $(f,k) \maps (x,a) \to (y,b)$ with $f \maps x \to y$ a morphism in $\X$, and $k \maps a \to (\M f)(b)$ a morphism in $\M x$;
    \item composition $(g, \ell) \circ (f, k) \maps (x,a) \to (y,b) \to (z,c)$ is given by 
    \begin{equation}
    \label{eq:comp_intM}
        (g, \ell) \circ (f, k) := (g \circ f, \mu_{f,g} \circ \M f(\ell) \circ k)
    \end{equation}
    \item identity map $id_{(x,a)} \maps (x,a) \to (x,a)$ is given by $id_x \maps x \to x$ in $\X$ and $id_a \maps a \to \M(1_x)(a) = a$ in $\M x$.
\end{itemize}
Visualize the composition rule \cref{eq:comp_intM} as follows.
\[
\begin{tikzcd}
    x
    \arrow[d, swap, "f"]
    \arrow[dd, bend left, "gf"]
    &
    a
    \arrow[r, "k"]
    &
    \M fb
    \arrow[r, "\M f(\ell)"]
    &
    \M f(\M g (c))
    \arrow[r, "\mu_{f,g}"]
    &
    \M (gf) (c)
    \\
    y
    \arrow[d, swap, "g"]
    &&
    b
    \arrow[r, "\ell"]
    \arrow[u, mapsto]
    &
    \M g (c)
    \arrow[u, mapsto]
    \\
    z
    &&&
    c
    \arrow[u, mapsto]
    \arrow[uur, mapsto]
\end{tikzcd}
\]
In the case where $\M$ is a strict functor, the composition formula simplifies to the following.
\begin{equation}
    \label{eq:comp_intMstrict}
    (g, \ell) \circ (f, k) = (g \circ f, \M f(\ell) \circ k)
\end{equation}
Notice that as sets, we have
\begin{equation}
\label{eq:hom}
    \inta\M((x,a),(y,b)) = \X(x,y) \times \M(x)(a,\M f(b)).
\end{equation}

\begin{expl}[Ring-module pairs]
\label{ex:mod2}
    The Grothendieck construction applied to the indexed category $\Mod \maps \Ring\op \to \Cat$ from \cref{ex:mod1} gives the category $\inta\Mod$ where objects are ring-module pairs, $(R, M)$ with $R$ a ring and $M$ an $R$-module, and morphisms are pairs $(\phi, \phi^\sharp) \maps (R, M) \to (S,N)$ where $\phi \maps R \to S$ is a ring homomorphism, and $\phi^\sharp \maps M \to \phi^*(N)$ is an $R$-module homomorphism. 
\end{expl}

\begin{lem}
\label{lem:invert}
    Maps $f \maps x \to y$ in $\X$ and $k \maps a \to \M f(b)$ in $\M x$ are isomorphisms if and only if $(f,k) \maps (x,a) \to (y,b)$ is an isomorphism in $\inta \M$. 
\end{lem}

The above result is more or less immediate. Note though that the inverse of $(f,k)$ is not $(f\inv, k\inv)$, as that map does not even have the right type. Instead, the inverse is $(f\inv, \M f\inv(k))$. 

\begin{expl}[Semidirect product]
\label{ex:action2}
    Let $A \maps G\op \to \Grp\hookrightarrow \Cat$ be as in \cref{ex:action1}, and let $\star_G$ and $\star_H$ denote the unique objects of $G$ and $H$ respectively. Then the Grothendieck construction applied to $A$ gives a category $\inta A$ with exactly one object $(\star_G, \star_H)$, and morphisms have the form $(g, h)$ with $g \in G$ and $h \in H$. By \cref{lem:invert}, $\inta A$ is a group. The composition rule \cref{eq:comp_intMstrict} specialized to this case is $(g_1,h_1) \circ (g_2,h_2) = (g_1 \circ g_2, g_2.h_1 \circ h_2)$. The group $\inta A$ is precisely the semidirect product $G \ltimes H$. 
\end{expl}

\begin{expl}[Arrow category]
\label{ex:arrowcat}
    Consider the indexed category of slices of \cref{ex:slice}. The Grothendieck construction $\int(\X/-)$ has for objects pairs $(x, f \maps y \to x)$ and for morphisms $(k, h) \maps (x, f \maps y \to x) \to (z, g \maps w \to z)$ with $k \maps x \to z$, and $h \maps f \to k^*g$. With some work, it can be seen that this category is equivalent to the \define{arrow category} of $\X$, denoted $\X^\to$, consisting of arrows of $\X$ for objects, and commutative squares for morphisms.
\end{expl}

\subsection{Fibrations}

A certain amount of information is lost when an indexed category is turned into its total category. Namely, there is no categorical property of $\inta\M$ which can tell you if two given objects were in the same fiber or not. This data can be preserved however by repackaging it into a functor. For any indexed category $\M$, there is a functor $P_\M \maps \inta \M \to \X$ given by projecting onto the first component for both objects and morphisms. This functor is especially nice. It admits a certain lifting property as we shall see. Such a functor is called a \emph{fibration}. 

Consider a functor $P \maps \A \to \X$. A morphism $\phi \maps a \to b$ in $\A$ over a morphism $f = P(\phi) \maps x \to y$ in $\X$ is called \define{cartesian} if and only if, for all $g \maps x' \to x$ in $\X$ and $\theta \maps a' \to b$ in $\A$ with $P \theta = f \circ g$, there exists a unique arrow $\psi \maps a' \to a$ such that $P \psi = g$ and $\theta = \phi \circ \psi$:
\begin{equation}
\begin{tikzcd}[column sep = huge]
    a'
    \arrow[drr, "\theta"]
    \arrow[dr, dashed, swap, "\exists!\psi"]
    \arrow[dd, dotted, bend right]
    \\&
    a
    \arrow[r, swap, "\phi"]
    \arrow[dd, dotted, bend right]
    &
    b
    \arrow[dd, dotted, bend right]
    &
    \text{in }\A
    \\
    x'
    \arrow[drr, "f \circ g = P \theta"]
    \arrow[dr, swap, "g"]
    \\&
    x
    \arrow[r, swap, "f = P \phi"]
    &
    y
    &
    \text{in }\X
\end{tikzcd}
\end{equation}
For $x \in \ob\X$, the \define{fibre of $P$ over $x$} written $\A_x$, is the subcategory of $\A$ which consists of objects $a$ such that $P(a) = x$ and morphisms $\phi$ with $P(\phi) = 1_x$. The functor $P \maps \A \to \X$ is called a \define{fibration} if and only if, for all $f \maps x \to y$ in $\X$ and $b \in \A_y$, there is an object $a \in \A_x$ and a cartesian morphism $\phi \maps a \to b$ with $P(\phi) = f$; it is called a \define{cartesian lift} of $f$ to $b$. The category $\X$ is then called the \define{base} of the fibration, and $\A $ its \define{total category}.

\begin{lem}
\label{lem:isocartesian}
    An isomorphism is always cartesian.
\end{lem}

A \define{fibred functor} $H \maps P \to Q$ between fibrations $P \maps \A \to \X$ and $Q \maps \B \to \X$ is given by a commutative triangle
\begin{equation}
\label{eq:fibredfunctor}
\begin{tikzcd}[row sep = huge]
    \A
    \arrow[rr, "H"]
    \arrow[dr, swap, "P"]
    &&
    \B
    \arrow[dl, "Q"]
    \\&
    \X
\end{tikzcd}
\end{equation}
where the top $H$ preserves cartesian liftings, meaning that if $\phi$ is $P$-cartesian, then $H\phi$ is $Q$-cartesian.
A \define{fibred natural transformation} is a natural transformation $\beta \maps H \To K$ such that $Q(\beta_a) = Pa$ for all objects $a \in \A$.
\begin{equation}
\label{eq:fibrednaturaltrans}
\begin{tikzcd}[row sep = huge]
    \A
    \arrow[rr, bend left, "H"]
    \arrow[rr, phantom, "\Downarrow \beta"]
    \arrow[rr, bend right, swap, "K"]
    \arrow[dr, swap, "P"]
    &&
    \B
    \arrow[dl, "Q"]
    \\&
    \X
\end{tikzcd}
\end{equation}

If $P\maps \A \to \X$ is a fibration, assuming the axiom of choice, we may select a cartesian arrow over each $f\maps x \to y$ in $\X$ and $b \in \A_y$, denoted by $\Cart(f, b) \maps f^*(b) \to b$. Such a choice of cartesian liftings is called a \define{cleaving} for $P$, which is then called a \define{cloven fibration}. Since identity maps are always cartesian, it is always possible to choose a cleaving where the lift of an identity is again an identity. This can be convenient, and will be used in the next section.

Suppose $P \maps \A \to \X$ is a cloven fibration. For any map $f \maps x \to y$ in the base category $\X$, the data of the cleaving can be used to define a \define{reindexing functor} between the fibre categories.
\begin{equation}
\label{reindexing}
    f^*\maps \A_y \to \A_x
\end{equation}
Send an object $b \in \A_y$ to the domain of the chosen cartesian lift of $f$ to $b$, $\Cart(f,b) \maps f^*(b) \to b$. Given a map $\psi \maps b \to b'$ in $\A_y$, define $f^*(\psi)$ by the universal property of cartesian maps as follows.
\[
\begin{tikzcd}[column sep = huge]
    f^*(b)
    \arrow[rr, "{\Cart(f,b)}"]
    \arrow[dr, dashed, swap, "\exists!f^*(\psi)"]
    \arrow[dd, dotted, bend right]
    &&
    b
    \arrow[ddd, dotted, bend left]
    \arrow[d, "\psi"]
    \\&
    f^*(b')
    \arrow[r, swap, "{\Cart(f,b')}"]
    \arrow[dd, dotted, bend right]
    &
    b'
    \arrow[dd, dotted, bend right]
    &
    \text{in }\A
    \\
    x
    \arrow[drr, "f"]
    \arrow[dr, swap, "id_x"]
    \\&
    x
    \arrow[r, swap, "f"]
    &
    y
    &
    \text{in }\X
\end{tikzcd}
\]
We leave it to the interested reader to verify this defines a functor. 

It can be verified by the cartesian universal property that $1_{\A_x} \cong (1_x)^*$ and that for composable morphism in the base category, $f^* \circ g^* \cong (g \circ f)^*$. Specific isomorphisms inhabiting these relations can be derived from a cleaving, a fact we shall employ in the next section. If these isomorphisms are equalities, we say the fibration is \define{split}.

Let $\M \maps \X\op \to \Cat$ be an $\X$-indexed category. Define the functor $P_\M \maps \inta \M \to \X$ by projecting onto the first component for both objects and morphisms. 
\begin{lem}
    The functor $P_\M$ defined above is always a fibration. Moreover, if $\M$ is a strict functor, then $P_\M$ is a split fibration.
\end{lem}

The cartesian lift of an arrow $f \maps x \to y$ in $\X$ to an object $(y,b)$ above $y$ is $(f, id_{\M fb}) \maps (x, \M fb) \to (y,b)$. Demonstrating that this satisfies the conditions of being cartesian involves no more than a calculation. 

\begin{expl}[Products]
\label{ex:products}
    Let $\X$ and $\Y$ be categories. The projection functor $\pi_\X \maps \X \times \Y \to \X$ is a fibration. Given a map $f \maps x \to x'$ in $\X$ and an object $(x',y)$ above $x'$, the map $(f, id_y) \maps (x,y) \to (x',y)$ is a cartesian lift of $f$ to $(x',y)$. The fiber over any object $x$ is equivalent to $\Y$. The reindexing functors provided by the cleaving above are always identity. Thus, $\X \times \Y$ is equivalent to the Grothendieck construction of \define{the constant indexed category} $\Delta\Y \maps \X\op \to \Cat$ given by $x \mapsto \Y$ for all $x \in \X$, and $f \mapsto id_\Y$ for all maps $f$ in $\X$.
\end{expl}

\begin{expl}[Split group extensions]
\label{ex:action3}
    Let $G$, $H$, and $A$ be as in \cref{ex:action1} and \cref{ex:action2}. The condition of the projection $p \maps \inta A = G \ltimes H \to G$ being a fibration essentially amounts to saying it is surjective, since all invertible maps in $\inta A$ are cartesian. This makes up the surjective part of the short exact sequence expressing $\inta A$ as an extension of $G$ by $H$.
    \[0\to H \xrightarrow{i} \inta A \xrightarrow{p} G \to 0
    \]
    The fact that $A$ is a strict functor tells us that the the compositor $\mu_{g_1, g_2} \maps Ag_1 \circ Ag_2 \To A(g_2g_1)$ is the identity, and thus we have a commutative diagram.
    \[
    \begin{tikzcd}
        A(g_2g_1)(\star) 
        \arrow[d, equals]
        \arrow[drr, "\Cart(g_2g_1)"]
        \\
        Ag_1(Ag_2(\star))
        \arrow[r, swap, "\Cart(g_1)"]
        &
        Ag_2(\star)
        \arrow[r, swap, "\Cart(g_2)"]
        &
        \star
    \end{tikzcd}
    \]
    Note that all the objects in this diagram are the unique object of $\inta A$. The interesting part of the diagram, as usual, is what it says about the morphisms. It says that $\Cart$ is in fact a group homomorphism $\Cart \maps G \to \inta A$, splitting the sequence. This recovers the well-known fact that split extensions are always semidirect products, whence the naming of split fibrations.
\end{expl}

\begin{expl}[Codomain fibration]
    Let $\X$ and $\X/-$ be as in \cref{ex:slice} and \cref{ex:arrowcat}. The projection $P\maps \int(\X/-) \simeq \X^\to \to \X$ maps an arrow to its codomain, and a commutative square to its lower edge. This fibration is known as the \define{codomain fibration} for obvious reasons, and the \define{fundamental fibration} because it ties together so many fundamental concepts of category theory.
\end{expl}

The indexed category $\M$ can be recovered up to isomorphism from the fibration $\M$. The fibre over an object $x$ in the base is equivalent to the category $\M(x)$. Identifying these, the reindexing functor $f^*$ induced by a map $f \maps x \to y$ in the base is precisely the functor $\M(f) \maps \M(y) \to \M(x)$. Moreover, every fibration comes from the Grothendieck construction of an indexed category, and every indexed category comes from a fibration. Indeed, the Grothendieck correspondence gives an equivalence between these two pieces of data. The idea of the following result was present in the work of Grothendieck where the eponymous construction was introduced \cite{Grothendieckcategoriesfibrees}. A detailed proof can be found in \cite{2DCats}.

\begin{thm}
\label{thm:Grothendieck}
    The Grothendieck construction specifies an equivalence of 2-categories $\ICat(\X) \simeq \Fib(\X)$, which restricts to an equivalence of categories $\ICat_s(\X) \simeq \Fib_s(\X)$.
\end{thm}

\section{Extending automorphism groups}
\label{sec:Auto}

In this section, we describe the relationship between the automorphism groups of the objects in a category $\X$, and the automorphism groups of the objects in $\inta\M$, where $\M$ is some indexed category. Much of this relationship is known to experts, and some detail can be found for example in \cite{BaezShulman}. In general, an object in a category can have both invertible and non-invertible endomorphisms, and so it is of general interest to discuss the endomorphism monoids. However, in the context of $\FI$-type categories, all endomorphisms are invertible, and so this reduces to the case of automorphism groups anyway. For some discussion of monoids and the Grothendieck construction, see \cite{NetworkModels, monoidGrothConst}.

Let $\M \maps \X\op \to \Cat$ be an indexed category, and $(x,a)$ an object in $\inta \M$. If we assume without loss of generality that $\M x$ is skeletal, \cref{lem:invert} tells us that an automorphism $(f,k) \maps (x,a) \to (x,a)$ consists of an automorphism $f \maps x \to x$ in $\X$, and an automorphism $k \maps a \to a$ in $\M x$. The automorphism group of $(x,a)$ in $\inta \M$ is precisely the category $\inta\M|_x$ where $\M|_x \maps \Aut(x)\op \to \Grp$ is the appropriate restriction of $\M$ to the subcategory $\Aut(x)$ of $\X$ consisting of automorphisms of $x$. 
Thus we can understand the automorphism groups of $\inta \M$ by first understanding what the Grothendieck construction does to pseudofunctors of the form $A \maps G \op \to \Grp$. Recall from \cref{ex:action1}, \cref{ex:action2}, and \cref{ex:action3} that if $A$ is a strict functor, then it is an action of $G$ on the group $A(\star)$, $\inta A$ is precisely the semidirect product, the fibration is the surjective part of the short exact sequence expressing $\inta A$ as an extension, and the fact that it is a split fibration corresponds exactly to the sequence being split.
This section is dedicated to the non-strict case.

\begin{prop}
    Conflating groups with one-object groupoids, surjective homomorphisms are precisely fibrations.
\end{prop}
\begin{proof}
    Let $p \maps E \to G$ be a fibration between groups. For each $g \in G$, there is a cartesian arrow $h$ in $E$ such that $p(h) = g$.
    
    Let $p \maps E \to G$ be a surjective group homomorphism, i.e.\ a full functor between one-object groupoids. For an element $g \in G$, there is an element $h \in E$ such that $f(h) = g$. By \cref{lem:isocartesian}, $h$ is $p$-cartesian.
\end{proof}

\subsection{The 2-category of groups}
The category of groups $\Grp$ is a full subcategory of $\Cat$, but $\Cat$ is naturally a 2-category, with natural transformations as the 2-morphisms. What then are natural transformations between group homomorphisms? Let $G$ and $H$ be groups, and $f,k \maps G \to H$ be homomorphisms. A natural transformation $\alpha \maps f \To k$ consists of a map in $H$ for each object in $G$. Since there is only one object in $G$, $\alpha$ has a single component, an element of $H$ which we also refer to as $\alpha$. This element must satisfy the naturality condition:
\[
\begin{tikzcd}
    \bullet 
    \arrow[r, "\alpha"]
    \arrow[d, swap, "f(g)"]
    &
    \bullet 
    \arrow[d, "k(g)"]
    \\
    \bullet 
    \arrow[r, swap, "\alpha"]
    &
    \bullet
\end{tikzcd}
\]
i.e.\ $\alpha \cdot f(g) = k(g) \cdot \alpha$ for all $g \in G$. In this scenario, $f$ and $k$ are related by an inner automorphism. We sometimes refer to such a natural transformation as an \define{intertwining element}. As a special case, a natural transformation of the form $\alpha \maps id_G \To id_G \maps G \to G$ is precisely an element in the center of $G$. The horizontal composite $\alpha_2 \ast \alpha_1 \maps f_2 \circ f_1 \To g_2 \circ g_1$ of natural transformations $\alpha_1 \maps f_1 \To k_1$ and $\alpha_2 \maps f_2 \To k_2$
\[
\begin{tikzcd}[column sep = large]
    G
    \arrow[r, bend left = 40, "f_1"]
    \arrow[r, phantom, "\Downarrow \alpha_1"]
    \arrow[r, bend right = 40, swap, "k_1"]
    &
    H
    \arrow[r, bend left = 40, "f_2"]
    \arrow[r, phantom, "\Downarrow \alpha_2"]
    \arrow[r, bend right = 40, swap, "k_2"]
    &
    K
\end{tikzcd}
\]
is given by $\alpha_2 \ast \alpha_1 = k_2(\alpha_1) \alpha_2$. Vertical composition of natural transformations $\alpha$ and $\beta$
\[
\begin{tikzcd}[column sep = huge]
    G
    \arrow[r, bend left=80, "f"', swap, ""{name = f}]
    \arrow[r, "k"description, swap, ""{name = g}, ""'{name = gg}]
    \arrow[r, bend right=80, ""{name = h}, "\ell"']
    &
    H
    \arrow[from = f, to = gg, Rightarrow, "\alpha"]
    \arrow[from = g, to = h, Rightarrow, "\beta"]
\end{tikzcd}\]
is just given by multiplication $\beta \circ \alpha = \beta\alpha$.

Note that this gives us for each pair of groups $G$ and $H$ a category $\Grp(G,H)$ where the objects are group homomorphisms $G \to H$, and morphisms are these intertwining elements, allowing us to consider $\Grp$ as a sub-2-category of $\Cat$. Since these intertwining elements have inverses, $\Grp(G,H)$ is actually a groupoid, making $\Grp$ a groupoid-enriched category, or (2,1)-category. It is also equivalent to the full sub-2-category of the (2,1)-category of groupoids consisting of the connected groupoids.

\begin{prop}
    A pseudofunctor $A \maps G\op \to \Grp$ consists of 
    \begin{itemize}
        \item a group $H$
        \item a function $A \maps G \to \Aut(H)$ (we still denote it like a right action when convenient, for example $Ag(h) = h.g$)
        \item a function $\phi \maps G \times G \to H$
    \end{itemize}
    such that
    \begin{itemize}
        \item $(h.g_1).g_2 = \phi_{g_1, g_2}\inv (h.(g_1g_2)) \phi_{g_1, g_2}$ for all $h \in H$
        \item for three elements $g_1, g_2, g_3 \in G$, we have $\phi_{g_3g_2,g_1}(\phi_{g_2,g_3}.g_1) = \phi_{g_3,g_2g_1}$
    \end{itemize}
\end{prop}
\begin{proof}
    As with a strict action, the unique object of $G$ is mapped to an object of the target category, namely $H$, and each element of $G$ is mapped to a morphism in the target category, an automorphism of $H$. For strict functors, preservation of composites and identity maps is a property, but for pseudofunctors, preservation of composites and identity maps is mediated by invertible 2-morphisms, intertwining elements in our case. As noted previously, the identity preservation data can be chosen without loss of generality to be trivial. The data of the compositor for $A$ is an intertwining element of $H$ for every pair of composable morphisms in $G$. Since every pair of maps is composable in $G$, this is captured by a function of the form $\phi \maps G\times G \to H$. The two conditions are precisely the unit law and the associativity law for a pseudofunctor.
\end{proof}

This is essentially a right action of $G$ on $H$, where the action laws hold up to conjugation by a special family of elements. One might want to call such a thing a ``twisted (right) action'' of $G$ on $H$ \cite{TwistedAction}.

\subsection{Twisted action from surjection}

How is a pseudofunctor $A \maps G\op \to \Grp$ constructed from a surjection of groups $p \maps E \to G$? Keep in mind that any surjection is part of a unique short exact sequence by taking the kernel $K:=\ker p$.
\[0 \to K \xrightarrow i E \xrightarrow p G \to 0\]
So $A$ must send the unique object $\star$ of $G$ to the group $K$, since $K$ is the fiber of $p$ over $\star$. Let $s \maps G \to E$ be a (set-theoretic) section of $p$. This provides a cleaving of $p$ as a fibration, with $\Cart(g, \star_E) = s(g)$ for each $g \in G$. We use this to define the pseudofunctors action on  morphisms by $Ag \maps K \to K$ by $Ag(k) = s(g)\inv k s(g)$. If $s$ happens to be a group homomorphism, making this extension split, then we would have
\begin{align*}
    Ag(Ah(k))
    &= s(g)\inv s(h)\inv k s(h) s(g)
    \\&= s(hg)\inv k s(hg)
    \\&= A(hg)(k).
\end{align*}
This tells us that $A$ is a strict functor, and we are done defining it. Otherwise, we must specify the compositor and unitor data. 

To specify the unitor $\phi_0 \maps A(id_\star) = A(e_G) \To id_K$, we only need one element $\phi_0$ of $K$ such that $Ae_G(k) = \phi_0\inv k\phi_0$. Let $\phi_0 := s(e_G)$. Without loss of generality, we can assume that $s(e_G) = e_E$, and thus $\phi_0 = e_E$. To specify the compositor, for each pair of elements $\sigma, \tau \in G$, we need an element $\phi_{\sigma, \tau} \in K$ such that $\phi_{\sigma, \tau} \maps A\sigma \circ A\tau \To A(\tau \sigma)$, meaning $\phi_{\sigma, \tau} A\sigma(A\tau(k)) = A(\tau\sigma)(k) \phi_{\sigma, \tau}$. Let 
\begin{align*}
    \phi_{\sigma, \tau} 
    &:= s(\tau\sigma)\inv s(\tau) s(\sigma).
\end{align*}

\begin{prop}
    Let $p \maps E \to G$ be a surjective group homomorphism. A set-theoretic section $s \maps G \to E$ determines a twisted (right) action $(A,\phi)$ of $G$ on $\ker p$ by $Ag(k) = s(g)\inv ks(g)$ and $\phi_{\sigma, \tau} = s(\tau\sigma)\inv s(\tau) s(\sigma)$. This is independent of the choice of section. Moreover, every twisted action of $G$ is determined by a surjection onto $G$ in this way.
\end{prop}
This is a direct corollary of \cref{thm:Grothendieck}.

\section{Fibrations of $\FI$-type categories}
\label{sec:GCFI}

In \cite{FItypeCats}, Gadish provides some purely categorical conditions on a category which are sufficient to prove some generalized representation stability results. Categories with these properties, which we recall below, are referred to as \emph{$\FI$-type categories}. In this section, we state sufficient conditions on an indexed category $\M \maps \X\op \to \Cat$ with base category of $\FI$-type for $\inta\M$ to also be of $\FI$-type. We then provide several examples of $\FI$-type categories, including those built using the Grothendieck construction.

\begin{defn}
\label{def:weakpushout}
    A \define{weak pushout} diagram is a pullback square 
    \[
    \begin{tikzcd}
        p
        \arrow[r, "\overline f_1"]
        \arrow[d, swap, "\overline f_2"]
        \arrow[dr, phantom, "\lrcorner", pos = 0.1]
        &
        c_1
        \arrow[d, "f_1"]
        \\
        c_2
        \arrow[r, swap, "f_2"]
        &
        d
    \end{tikzcd}
    \]
    which is universal among pullback squares that agree on the top and left maps: for any pullback square 
    \[
    \begin{tikzcd}
        p
        \arrow[r, "\overline f_1"]
        \arrow[d, swap, "\overline f_2"]
        \arrow[dr, phantom, "\lrcorner", pos = 0.1]
        &
        c_1
        \arrow[d, "h_1"]
        \\
        c_2
        \arrow[r, swap, "h_2"]
        &
        z
    \end{tikzcd}\]
    there exists a unique morphism $h \maps d \to z$ that makes the diagram
    \[
    \begin{tikzcd}
        p
        \arrow[r, "\overline f_1"]
        \arrow[d, swap, "\overline f_2"]
        &
        c_1
        \arrow[d, "f_1"]
        \arrow[ddr, bend left, "h_1"]
        \\
        c_2
        \arrow[r, swap, "f_2"]
        \arrow[drr, swap, bend right, "h_2"]
        &
        d
        \arrow[dr, dashed, "h"]
        \\&&z
    \end{tikzcd}
    \]
     commute. We call $d$ the \define{weak pushout object} and denote it by $c_2 \sqcup_p c_2$. The unique map $h$ induced from a pair of maps $h_i \maps c_i \to z$ is denoted by $h_1 \sqcup_p h_2$. A functor $F \maps \C \to \D$ is said to \define{preserve weak pushouts} if the image of a weak pushout square is also a weak pushout square.
\end{defn}

Note that a pushout is automatically a weak pushout. This restricted notion of pushout also appears in \cite{Devissage}, in the context of algebraic $K$-theory.

\begin{defn}[\cite{FItypeCats}]
\label{def:FItype}
    A category $\X$ is of \define{$\FI$-type} if it satisfies the following conditions.
    \begin{enumerate}
        \item $\X$ is locally finite, i.e.\ all hom-sets are finite
        \item every morphisms is a monomorphism
        \item every endomorphism is an isomorphism
        \item for every pair of objects $x$ and $y$, the group of automorphisms $\Aut_\X(y)$ acts transitively on the set $\Hom_\X(x, y)$
        \item for every object $y$ there exist only finitely many isomorphism classes of objects $x$ for which $\Hom_\X(x, y) \neq \emptyset$ (we denote this by $x \leq y$)
        \item $\X$ has pullbacks
        \item $\X$ has weak pushouts
    \end{enumerate}
\end{defn}

\begin{expl}
    Every locally finite groupoid is $\FI$-type. This turns out not to be a terribly interesting fact on its own. For instance, Gadish's results about $\FI$-type categories when applied to $\FB$ essentially says that a subrepresentation of a finite dimensional representation of a symmetric group is also finite dimensional. 
\end{expl}

\begin{thm}[\cite{FItypeCats}, Thm.\ C]
    If $\C$ is a category of $\FI$-type, then the category of $\C$-modules over $\CC$ is Noetherian. That is, every submodule of a finitely generated $\C$-module is itself finitely generated. 
\end{thm}

We require a technical definition before we can state the main theorem.

\begin{defn}
    We say that an indexed category $\M \maps \X\op \to \Cat$ is \define{weakly reversible} if for each map $f \maps x \to y$ in $\X$, there is a weak pushout preserving functor $f_! \maps \M x \to \M y$ such that $f_!f^*$ is identity on objects, along with a natural transformation $\eta^f \maps id_{Fx} \To f^*f_!$.
\end{defn}

\begin{thm}
\label{thm:main}
    Let $\X$ be an $\FI$-type category, and let $\M \maps \X\op \to \Cat$ be an indexed category. Then $\inta\M$ is an $\FI$-type category and the related fibration preserves pullbacks and weak pushouts if
    \begin{enumerate}
        \item the fibers are $\FI$ type categories
        \item for every endomorphism $f \maps x \to x$ and object $a$ in the fiber over $x$, every map $a \to \M f(a)$ in $\M x$ is invertible
        \item the inclusions $\M x \hookrightarrow \inta \M$ preserve pullbacks
        \item $\M$ is weakly reversible
    \end{enumerate} 
    for all objects $x,y$ and morphisms $f \maps x \to y$ in $\X$. 
\end{thm}

This theorem is proven in the next section.

\begin{expl}[Finite products of $\FI$-type categories]
    For categories $\X$ and $\Y$, recall the constant indexed category $\Delta \Y \maps \X\op \to \Cat$ from \cref{ex:products}. If $\X$ and $\Y$ are of $\FI$-type, then $\Delta \Y$ clearly satisfies the conditions of \cref{thm:main}. Thus products of $\FI$-type categories are also of $\FI$-type.
\end{expl}

\begin{expl}[$FI_G$]
    Let $G$ be a group, and $G^\bullet \maps \FI\op \to \Grp$ be the indexed group defined as follows. For an object $n \in \FI$, $n \mapsto G^n$. A morphism $f \colon n \to m$ is mapped to the pullback $f^\ast \colon G^m \to G^n$. The symmetric group $S_n$ acts on $G^n$ by permuting the copies of $G$. The Grothendieck construction gives a category $\inta G^\bullet$ with objects $(n, \star_n)$ with $n \in\FI$ and $\star_n$ the unique object in $G^n$, and morphisms $(f,g) \maps (n, \star_n) \to (m, \star_m)$ with $f \maps n \to m$ a morphism in $\FI$, and $g$ an element of the group $G^n$.
    The category $\inta G^\bullet$ is equivalent to $\FI_G$.
    The functor $\FI_G \to \FI$ which simply forgets the $G$-decorations is a fibration.
\end{expl}

\begin{prop}[$FI_G$]
    Let $G$ and $H$ be finite groupoids. Then $\FI_{G \sqcup H} \simeq \FI_G \times \FI_H$.
\end{prop}
\begin{proof}
    Define a functor $\FI_G \times \FI_H \to \FI_{G \sqcup H}$ by sending an object $(m,n)$ to $m+n$, and a morphism $((f, g_1, \dots, g_m), (j, h_1, \dots h_n))$ to $(f+j, g_1, \dots, g_m, h_1, \dots h_n)$. This is clearly essentially surjective. 
\end{proof}

\begin{expl}[Block permutations]
    Let $\M \maps \FI\op \to \Cat$ be given by $n \mapsto \FI^n$ on objects, and for an injection $f \maps n \to p$, define $f^* \maps \FI^p \to \FI^n$ by $f^*(q_1, \dots, q_p) = (q_{f(1)}, \dots, q_{f(n)})$. What is $\inta \M$ like? The objects are tuples $(n, m_1, \dots, m_n)$ of finite sets. This is an $n$-part partition of the set $\sum^n m_i$. A map $(n, m_1, \dots, m_n) \to (p, q_1, \dots, q_p)$ consists of an injection $f \maps n \to p$, and $(g_1, \dots, g_n) \maps (m_1, \dots, m_n) \to f^*(q_1, \dots, q_p)$, which is the same as a family of injections $g_i \maps m_i \to q_{f(i)}$.
    
    We naturally get a fibration $\inta\M \to \FI$ by projecting onto the first component, but we consider now another functor $\inta\M \to \FI$. Define a functor $c \maps \inta\M \to \FI$ by $(n, m_1, \dots, m_n) \mapsto \sum^n m_i$, and similar on morphisms. This is full and essentially surjective, but not faithful, which tells us that it is forgetting purely stuff.
\end{expl}

\section{Proof of \cref{thm:main}}
\label{sec:proofs}

In this section we provide a proof for \cref{thm:main}. We prove the corresponding statement for each property in \cref{def:FItype} separately to aid the reader that is only concerned with a subset of the properties.

\subsection{Locally finite}

\begin{lem}[Locally finite]
\label{lem:locfin}
    Let $\M \maps \X\op \to \Cat$ be an indexed category with $\X$ locally finite. Then $\inta\M$ is locally finite if and only if $\M(x)$ is locally finite for each object $x \in \X$.
\end{lem}
\begin{proof}
    This follows immediately from \cref{eq:hom}.
\end{proof}

\subsection{Monomorphisms}

Nothing interesting is needed for the condition of every map being a monomorphism. If the base has the property, then the total category will have the property if and only if the fibers all have the property. Going from total to fibers follows from the fact that a monomorphism in a category remains a monomorphism in all subcategories. The converse is a direct check.

\begin{lem}[Every map is a monomorphism]
    Let $\X$ be a category where every map is a monomorphism, and $\M \maps \X\op \to \Cat$ be an indexed category. Then every map in $\inta \M$ is a monomorphism if and only if each map in the category $\M x$ is a monomorphism for each $x \in \X$.
\end{lem}
\begin{proof}
    Let $\M \maps \X\op \to \Cat$ be such that every map in $\M x$ is a monomorphism. Let $(f,k) \maps (x,a) \to (y,b)$ and $(g_1, \ell_1), (g_2, \ell_2) \maps (z, c) \to (x,a)$ be such that $(f,k) \circ (g_1, \ell_1) = (f,k) \circ (g_2, \ell_2)$. Then we get $f \circ g_1 = f \circ g_2$, and since every map in $\X$ is a monomorphism, we get $g_1 = g_2$. We also get $\M g_1(k) \circ \ell_1 = \M g_2(k) \circ \ell_2 = \M g_1(k) \circ \ell_2$, and since every map in $\M z$ is a monomorphism, we have $\ell_1 = \ell_2$. Thus $(f,k)$ is a monomorphism.
    
    To prove the converse, consider instead a fibration $P \maps \A \to \X$ such that every map in $\A$ or $\X$ is a monomorphism. Let $f \maps a \to b$ and $g_1, g_2 \maps c \to a$ be maps in $\A_x$ such that $f \circ g_1 = f \circ g_2$. Since these are maps in $\A$, $f$ is a monomorphism, and we get $g_1 = g_2$. 
\end{proof}

\subsection{Endomorphisms are invertible}

A category where every endomorphism is an isomorphism is called an \define{EI category}. This condition is equivalent to the condition that every morphism between isomorphic objects is invertible. \cref{prop:EIicat} gives sufficient conditions for the Grothendieck construction of an indexed category $\M \maps \X\op \to \Cat$ with an EI base category to be an EI category. \cref{prop:EIfib} is the converse, though phrased in the language of fibrations. So these conditions are both necessary and sufficient.

\begin{prop}
\label{prop:EIicat}
    Let $\X$ be an EI category, $\M \maps \X\op \to \Cat$ an indexed category such that 
    \begin{itemize}
        \item $\M x$ is EI for all objects $x \in \X$ 
        \item for every endomorphism $f \maps x \to x$ and object $a$ in the fiber over $x$, every map $a \to \M f(a)$ in $\M x$ is invertible
    \end{itemize}
    Then $\inta \M$ is EI.
\end{prop}
\begin{proof}
    Assume that $\M x$ is an EI category for each object $x \in \X$ and for an endomorphism $f \maps x \to x$ every map $a \to \M f(a)$ in $\M x$ is invertible. Let $(f,k) \maps (x,a) \to (x,a)$ be an endomorphism in $\inta\M$. Then $f \maps x \to x$ is an endomorphism, and thus invertible, and $k \maps a \to \M f(a)$ must be invertible by assumption. Since $f$ and $k$ are both invertible, $(f,k)$ is invertible.
\end{proof}

It is important to notice here that we are not asking for $a$ to be isomorphic to $\M f(a)$. It is possible that there are no maps of the form $a \to \M f(a)$, in which case, the above condition would be vacuously true. Take for example $X$ to be the one-object groupoid whose automorphism group of the unique object is $\ZZ$, and take $\M$ to assign to the unique object the discrete category with object set $\ZZ$, and the action is given by translation. The category $\inta\M$ is a groupoid, and thus it is EI, but usually we do not have $a \cong \M f(a)$ or indeed any maps between them in either direction.

\begin{lem}
\label{lem:invertinfiber}
    Let $\A$ and $\X$ be EI categories, and $P \maps \A \to \X$ be a fibration. Let $\phi \maps a \to b$ be a map in $\A_x$ which has an inverse $\phi\inv$ in $\A$. Then $\phi\inv$ is also in the fiber.
\end{lem}
\begin{proof}
    $P(\phi\inv) = P(\phi)\inv = id_x\inv = id_x$.
\end{proof}

\begin{lem}
\label{lem:liftinvmono}
    A cartesian lift of an invertible map is always a monomorphism.
\end{lem}
\begin{proof}
    Let $P \maps \A \to \X$ be a fibration, $f \maps x \to y$ be an isomorphism in $\X$, and $\phi \maps a \to b$ a cartesian lift of $f$. Consider two maps $\psi_1, \psi_2 \maps a' \to a$ such that $\phi \circ \psi_1 = \phi \circ \psi_2$. Then we have the following setup.
    
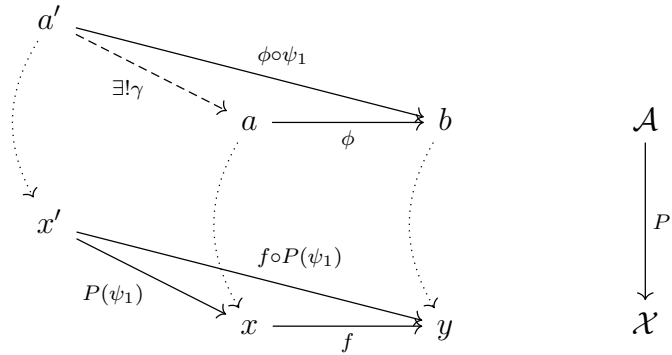
\[
    \begin{tikzcd}[column sep = huge]
        a'
        \arrow[drr, "\phi \circ \psi_1"]
        \arrow[dr, dashed, swap, "\exists!\gamma"]
        \arrow[dd, dotted, bend right]
        \\&
        a
        \arrow[r, swap, "\phi"]
        \arrow[dd, dotted, bend right]
        &
        b
        \arrow[dd, dotted, bend right]
        &
        \A
        \arrow[dd, "P"]
        \\
        x'
        \arrow[drr, "f \circ P(\psi_1)"]
        \arrow[dr, swap, "P(\psi_1)"]
        \\&
        x
        \arrow[r, swap, "f"]
        &
        y
        &
        \X
    \end{tikzcd}\]
    By the definition of $\phi$ being cartesian, there is a unique map $\gamma \maps a' \to a$ in $\A$ such that $\phi \circ \gamma = \phi \circ \psi_1$ and $P(\gamma) = P(\psi_1)$. Clearly $\psi_1$ satisfies these conditions as well, so $\gamma = \psi_1$. Note that 
    \begin{align*}
        P(\psi_1) 
        &= f\inv \circ f \circ P(\psi_1) 
        \\&= f\inv \circ P(\phi) \circ P(\psi_1) 
        \\&= f\inv \circ P(\phi \circ \psi_1)
        \\&= f\inv \circ P(\phi \circ \psi_2)
        \\&= f\inv \circ P(\phi) \circ P(\psi_2)
        \\&= f\inv \circ f \circ P(\psi_2)
        = P(\psi_2).
    \end{align*}
    Thus $\psi_1=\psi_2$.
\end{proof}

\begin{prop}
\label{prop:EIfib}
    Let $P\maps \A \to \X$ be a fibration with $\A$ and $\X$ being EI categories. Then for each object $x \in \X$, the fiber $\A_x$ is an EI category, and for an endomorphism $f \maps x \to x$ in $\X$ and an object $a \in \A_x$, each map $a \to f^*a$ in $\A_x$ is invertible.
\end{prop}
\begin{proof}
    Let $k \maps a \to a$ be an endomorphism in $\A_x$. Remember that being in the fiber means that $P(k) = id_x$. Since $k$ is an endomorphism in $\A$, it has an inverse. The map $k\inv$ is in the fiber $\A_x$ because $P(k\inv) = P(k)\inv = id_k\inv = id_x$. So the fibers are EI.
    
    Let $\ell \maps a \to f^*a$ be a map in $\A_x$, and let $\phi \maps f^*a \to a$ denote the cartesian lift of $f$ to $a$. Then the composite $a \xrightarrow \ell f^*a \xrightarrow \phi a$ is an endomorphism in $\A$, and thus invertible. Let $\psi \maps a \to a$ denote the inverse.  
    We claim that $\psi \phi$ is inverse to $\ell$. 
    Notice $\phi \ell \psi \phi =\phi$, and since $\phi$ is a monomorphism by \cref{lem:invertinfiber}, then we have $\ell \psi \phi = id_{f^*a}$ as desired. By \cref{lem:liftinvmono}, since $\ell$ is invertible in $\A$, it is invertible in $\A_x$.
\end{proof}

\subsection{Increasing}

We say that a category is \define{increasing} if each object only has finitely many isomorphism classes of objects which map into it.

\begin{lem}
\label{lem:increasing}
    Let $\X$ be an increasing category, and $\M \maps \X\op \to \Cat$ an indexed category. Then $\inta\M$ is increasing if and only if each fiber is increasing.
\end{lem}
\begin{proof}  
    Assume that each fiber is increasing. To have $(x,a) \leq (y,b)$, we must have a map $(f,k) \maps (x,a) \to (y,b)$, which consists of a map $f \maps x \to y$ and a map $k \maps a \to \M f(b)$. For a fixed $(y,b)$, there are finitely many choices for $x$, and finitely many choices for $a$. Thus $\inta\M$ is increasing.
    
    Let $P \maps \A \to \X$ be a fibration where $\A$ and $\X$ are increasing. Let $x$ be an object in $\X$, and $a \leq b$ in $\A_x$, i.e.\ there is a map $f \maps a \to b$ in $\A_x$. Since $f$ is a map in $\A$, there are finitely many choices for $a$, and so the choices are the choices within $\A$ intersected with $\A_x$, which is still finite. 
\end{proof}

\subsection{Transitive hom-action}

Following Gan--Li \cite{EICat}, we say that a category is \define{transitive} if for each pair of objects $x,y$, the monoid $End(y)$ acts transitively on $Hom(x,y)$. Keep in mind that in EI categories, and thus in $\FI$-type categories, the endomorphism monoids are the same as the automorphism groups.

\begin{lem}[Transitive]
\label{lem:transitive}
    Let $\X$ be a transitive category, and $\M \maps \X\op \to \Cat$ be an indexed category. Then $\inta\M$ is transitive if and only if the fibers are transitive and given maps $k_1 \maps a \to \M f_1(b)$ and $k_2 \maps a \to \M f_2(b)$ in $\M x$, there exists a map $\ell \maps b \to \M g(b)$ which makes the following diagram commute, where $g \maps y \to y$ is the map such that $f_1 = g \circ f_2$.
    \[
    \begin{tikzcd}
        a
        \arrow[r, "k_1"]
        \arrow[d, swap, "k_2"]
        &
        \M f_1(b)
        \\
        \M f_2(b)
        \arrow[r, swap, "\M f_2(\ell)"]
        &
        \M f_2 \M g(b)
        \arrow[u, "\M_{g, f_2}"', "\sim"]
    \end{tikzcd}
    \]
\end{lem}
\begin{proof}
    Assume the latter condition. Let $(f_1,k_1), (f_1,k_1) \maps (x,a) \to (y,b)$. Since $f_1, f_2 \maps x \to y$ in $\X$, then there is a map $g \maps y \to y$ such that $f_1 = g \circ f_2$. Then by our assumption, we get a map $\ell \maps b \to \M g(b)$, which combined with $g$ gives a map $(g, \ell) \maps (y,b) \to (y,b)$. 
    \begin{align*}
        (f_2, k_2) \circ (g, \ell) 
        &= (f_2 \circ g, \M g(k_2) \circ \ell \circ \M_{g,f_2}) 
        \\&= (f_1, k_1).
    \end{align*}
    
    Assume that $\inta\M$ is transitive. Let $f_1, f_2 \maps x \to y$ in $\X$, $k_1 \maps a \to \M f_1(b)$ in $\M x$, $k_2 \maps a \to \M f_2(b)$ in $\M x$. These assemble into maps $(f_1, k_1), (f_2, k_2) \maps (x,a) \to (y,b)$ in $\inta\M$. Since $\inta\M$ is transitive, there is a map $(g, \ell) \maps (y,b) \to (y,b)$ such that $(f_1, k_2) = (g, \ell) \circ (f_2, k_2)$. This $\ell$ is the desired map, and satisfies the necessary equation.
\end{proof}

\subsection{Pullbacks}

This is the first section in which we will consider properties of the fibration as well as the total category. Indeed, all previous conditions in the definition of $\FI$-type category are properties which are preserved by any functor automatically.

\begin{thm}[\cite{Grayfibredandcofibred} Theorem 4.2]
    Let $P \maps \A \to \X$ be a fibration, and let $\X$ have $J$-limits, for some small category $J$. Then $\A$ has $J$-limits and $P$ preserves them if and only if
    \begin{enumerate}
        \item each fibre $\A_x$ has $J$-limits
        \item the inclusions $\A_x \hookrightarrow \A$ preserve $J$-limits.
    \end{enumerate}
\end{thm}

For completeness, we include a proof of the case of pullbacks, where $J$ is set to be the category $\bullet \rightarrow \bullet \leftarrow \bullet$ (identity maps omitted).

\begin{proof}[Proof for pullbacks]
    Assume the two conditions above. Take a diagram 
    \[
    \begin{tikzcd}
        &
        b
        \arrow[d, "\psi"]
        \\
        a
        \arrow[r, "\phi", swap]
        &
        c
    \end{tikzcd}
    \]
    in $\A$. 
    Apply $P$ to map it into $\X$, and then take a pullback. 
    \[
    \begin{tikzcd}
        x
        \arrow[r, dashed, "g"]
        \arrow[d, dashed, swap, "f"]
        \arrow[dr, phantom, "\lrcorner", pos = 0.1]
        &
        Pb
        \arrow[d, "P\psi"]
        \\
        Pa
        \arrow[r, "P\phi", swap]
        &
        Pc
    \end{tikzcd}
    \]
    Factor $\phi$ and $\psi$ into their vertical and horizontal parts, then reindex along $f$ and $g$ to obtain the following diagram.
    \[
    \begin{tikzcd}
        &&&
        d
        \arrow[dll, dashed]
        \arrow[drr, dashed]
        \arrow[dddd, phantom, "\lrcorner"{rotate = -45}, pos = 0.1]
        \\&
        f^*a
        \arrow[dl, "{\Cart(f,a)}", sloped]
        \arrow[dr, "f^*\phi_v"]
        &&&&
        g^*b
        \arrow[dl, "g^*\psi_v", swap]
        \arrow[dr, "{\Cart(f,b)}", sloped]
        \\
        a
        \arrow[dr, "\phi_v", swap]
        &&
        f^*P\phi^*c
        \arrow[dl, "{\Cart(f, P\phi^*c)}"', sloped]
        \arrow[d, "\sim", sloped]
        &&
        g^*P\psi^*c
        \arrow[dr, "{\Cart(g, P\psi^*c)}"', sloped]
        \arrow[d, "\sim", swap, sloped]
        &&
        b
        \arrow[dl, "\psi_v"]
        \\&
        P\phi^*c
        \arrow[drr, swap, "\psi_h = {\Cart(P\phi,c)}", bend right = 15]
        &
        (P\phi f)^*c
        \arrow[dr, "{\Cart(P\phi f,c)}"description]
        \arrow[rr, equals]
        &&
        (P\psi g)^*c
        \arrow[dl, "{\Cart(P\psi g,c)}"description]
        &
        P\psi^*c
        \arrow[dll, "\psi_h = {\Cart(P\psi,c)}", bend left = 15]
        \\&&&
        c
    \end{tikzcd}
    \]
    We obtain the object $d$ and the dashed arrows above by taking the pullback of $f^*\phi_v$ and $g^*\psi_v$ in $\A_x$. By the second assumed condition, this is also a pullback in $\A$. We claim that the outer frame of the above diagram is a pullback square of our original diagram. 
    
    Now we prove the universal property. Let $q\in \A$ and $\alpha \maps q \to a$ and $\beta \maps q \to b$ be maps such that $\phi \circ \alpha = \psi \circ \beta$. We can factor $\alpha$ and $\beta$ through the cartesian maps as follows. Apply $P$ to the competitor diagram, and obtain the map $h \maps Pq \to x$ by universal property.
    \[
    \begin{tikzcd}
        Pq
        \arrow[drr, bend left, "P\beta"]
        \arrow[dr, dashed, "\exists!h"]
        \arrow[ddr, bend right, swap, "P\alpha"]
        \\&
        x
        \arrow[dr, phantom, pos = 0.1, "\lrcorner"]
        \arrow[r, "g"]
        \arrow[d, swap, "f"]
        &
        Pb
        \arrow[d, "P\psi"]
        \\&
        Pa
        \arrow[r, swap, "P\phi"]
        &
        Pc
    \end{tikzcd}
    \]
    We obtain the factorization of $\alpha$ through the cartesian lift $\Cart(f,a)$ by the cartesian property.
    \[
    \begin{tikzcd}[column sep = huge]
        q
        \arrow[drr, "\alpha"]
        \arrow[dr, dashed, swap, "\exists!\overline \alpha"]
        \arrow[dd, dotted, bend right]
        \\&
        f^*a
        \arrow[r, swap, "{\Cart(f,a)}"]
        \arrow[dd, dotted, bend right]
        &
        a
        \arrow[dd, dotted, bend right]
        \\
        Pq
        \arrow[drr, "P\alpha"]
        \arrow[dr, swap, "h"]
        \\&
        x
        \arrow[r, swap, "f"]
        &
        Pa
    \end{tikzcd}
    \]
    Similarly, we obtain a map $\overline \beta \maps q \to g^*b$. Now we have a competitor diagram, and thus obtain the map $\eta \maps q \to d$ as follows.
    \[
    \begin{tikzcd}
        q
        \arrow[drr, bend left, "\overline\beta"]
        \arrow[dr, dashed, "\exists!\eta"]
        \arrow[ddr, bend right, swap, "\overline \alpha"]
        \\&
        d
        \arrow[r, "\overline \gamma"]
        \arrow[d, swap, "\overline \delta"]
        &
        g^*b
        \arrow[d, "g^*\psi_v"]
        \\&
        f^*a
        \arrow[r, swap, "f^*\phi_v"]
        &
        f^*P\phi^*c
        \arrow[r, equals]
        &
        g^*P\psi^*c
    \end{tikzcd}
    \]
    A computation shows this makes the necessary diagrams commute, and uniqueness follows from the universal property from which the map $\eta$ was originally derived, as well as the lifting properties from which $\overline \alpha$ and $\overline \beta$ were derived.

    Assume $\A$ has pullbacks and $P$ preserves pullbacks. Let $x \in \X$ and let 
    \[
    \begin{tikzcd}
        &
        b
        \arrow[d, "g"]
        \\
        a
        \arrow[r, swap, "f"]
        &
        c
    \end{tikzcd}
    \]
    be a diagram in $\A_x$. This is also a diagram in $\A$, so we can take its pullback there.
    \begin{equation}
    \label{pullback}
    \begin{tikzcd}
        d
        \arrow[d, swap, dashed, "p"]
        \arrow[r, "q", dashed]
        \arrow[dr, phantom, pos = 0.1, "\lrcorner"]
        &
        b
        \arrow[d, "g"]
        \\
        a
        \arrow[r, swap, "f"]
        &
        c
    \end{tikzcd}
    \end{equation}
    Apply $P$ to this square.
    \[
    \begin{tikzcd}
        Pd
        \arrow[d, swap, "Pp"]
        \arrow[r, "Pq"]
        \arrow[dr, phantom, pos = 0.1, "\lrcorner"]
        &
        Pb
        \arrow[d, "Pg"]
        \\
        Pa
        \arrow[r, swap, "Pf"]
        &
        Pc
    \end{tikzcd}
    =
    \begin{tikzcd}
        Pd
        \arrow[d, swap, "Pp"]
        \arrow[r, "Pq"]
        \arrow[dr, phantom, pos = 0.1, "\lrcorner"]
        &
        x
        \arrow[d, "id_x"]
        \\
        x
        \arrow[r, swap, "id_x"]
        &
        x
    \end{tikzcd}
    \]
    Since $P$ preserves pullbacks, then $Pd$ and $Pp$ and $Pq$ must form the pullback of the constant diagram at $x$. Thus $Pd=x$, $Pp = Pq = id_x$, and the entire square (\ref{pullback}) is in the fiber $\A_x$.
    
    Consider the following competitor diagram in $\A_x$, and the map $h \maps e \to d$ derived from the universal property of pullbacks in $\A$.
    \[
    \begin{tikzcd}
        e
        \arrow[drr, bend left, "s"]
        \arrow[dr, dashed, "\exists!h"]
        \arrow[ddr, bend right, swap, "r"]
        \\&
        d
        \arrow[r, "q"]
        \arrow[dr, phantom, pos = 0.1, "\lrcorner"]
        \arrow[d, swap, "p"]
        &
        b\arrow[d, "g"]
        \\&
        a
        \arrow[r, swap, "f"]
        &
        c
    \end{tikzcd}
    \]
    \emph{A priori} we do not know the map $h$ is in the fiber. It could be the case that $Ph$ is a non-trivial endomorphism of $x$.
    \begin{align*}
        P(h)
        &= id_x P(h)
        \\&= P(p)P(h)
        = P(ph)
        \\&= P(r)
        = id_x
    \end{align*}
    So $h$ is in $\A_x$. Uniqueness follows \emph{a fortiori} from uniqueness in $\A$. The fact that the inclusion $\A_x \hookrightarrow \A$ preserves pullbacks follows from the fact that we constructed the pullback in $\A$ to begin with.
\end{proof}

\subsection{Weak pushouts}

Given a fibration of categories with weak pushouts, where the functor preserves weak pushouts, it is straightforward to show that the fibers actually have weak pushouts. Include your diagram in the fiber into the total category, take weak pushout in the total category, and then factor the maps into their vertical parts. The resulting maps end up having equal codomain, and they form the legs of the desired square.

The other direction is more difficult. In the strong scenario where we have left adjoints to each reindexing functor, we can adjoint the fiber components of the legs, and pushforward into the fiber over the weak pushout of the base components. Then we can take weak pushout there, and adjoint back to get maps of the right type. This is a special case of a construction that works for any colimit. However, such left adjoints do not exist even in the case of $\FI_G$. Indeed, an adjunction between $G^n$ and $G^m$ would would be an isomorphism. This is the only subsection in this section of the paper with only sufficient conditions, not necessary and sufficient.

\begin{defn}
    We say that an indexed category $\M \maps \X\op \to \Cat$ is \define{weakly reversible} if for each map $f \maps x \to y$ in $\X$, there is a weak pushout preserving functor $f_! \maps \M x \to \M y$ such that $f_!f^*$ is identity on objects, along with a natural transformation $\eta^f \maps id_{\M x} \To f^*f_!$.
\end{defn}

\begin{prop}
    Let $\M \maps \X\op \to \Cat$ be a locally reversible indexed category with $\X$ and $\M x$ having weak pushouts for each $x \in \X$. Then $\inta \M$ has weak pushouts. 
\end{prop}
\begin{proof}
    Let $(z,c) \xleftarrow{(g,\ell)} (x,a) \xrightarrow{(f,k)} (y,b)$ be a diagram in $\inta \X$. Take the weak pushout of the base components of this diagram in $\X$. Let $w$ denote the weak pushout object.
    \[
    \begin{tikzcd}
        x
        \arrow[r, "f"]
        \arrow[d, swap, "g"]
        &
        y
        \arrow[d, dashed, "h"]
        \\
        z
        \arrow[r, swap, "j", dashed]
        &
        w
    \end{tikzcd}
    \]
    Consider the map $\overline k$ defined as the composite $f_!(a) \xrightarrow{f_!k}f_!f^*(b) = b$ in $\M y$, and the analogous map $\overline \ell$ defined as the composite $g_!(a) \xrightarrow{g_!(\ell)} g_!g^*(c) = c$ in $\M z$. Push these forward to $\M z$ by $h_!$ and $j_!$ respectively, then take the weak pushout in $\M z$. 
    \[
    \begin{tikzcd}
        &
        h_!f_!(a) 
        \arrow[r, "h_!(\overline k)"]
        \arrow[dl, equals]
        &
        h_!(b)
        \arrow[dd, dashed, "\overline m"]
        \\
        j_!g_!(a)
        \arrow[d, "j_!(\ell)"']
        \\
        j_!(c)
        \arrow[rr, dashed, "\overline n"']
        &&
        d
    \end{tikzcd}
    \]
    Adjoint the maps $\overline m$ and $\overline n$ to $\M y$ and $\M z$ respectively by defining the map $m \maps b \to h^*(d)$ to be the composite $b \xrightarrow{\eta^h_b} h^*h_! (b) \xrightarrow{h^*(\overline m)} h^*(d)$, and similarly the map $n \maps c \to j^*(d)$ to be the composite $c \xrightarrow{\eta^j_c} j^*j_!(c) \xrightarrow{j^*(\overline n)} j^*(d)$. We claim the following diagram is a weak pushout square.
    \[
    \begin{tikzcd}
        (x,a)
        \arrow[r, "{(f,k)}"]
        \arrow[d, "{(g,\ell)}"']
        &
        (y,b)
        \arrow[d, "{(h,m)}"]
        \\
        (z,c)
        \arrow[r, "{(j,n)}"']
        &
        (w,d)
    \end{tikzcd}
    \]
    The base component of this diagram is a weak pushout square by construction. The fiber component
    \[
    \begin{tikzcd}
        a
        \arrow[r, "k"]
        \arrow[d, "\ell"']
        &
        f^*b
        \arrow[r, "f^*m"]
        &
        f^*h^*d
        \arrow[d, "\mu_{f,h}"]
        \\
        g^*c
        \arrow[d, "g^*n"']
        &&
        (hf)^*d
        \\
        g^*j^*d
        \arrow[r, "\mu_{g,j}"']
        &
        (jg)^*d
        \arrow[ur, equals]
    \end{tikzcd}
    \]
    is also a weak pushout square because the pushforward functors are assumed to preserve weak pushouts. Concluding the desired universal property from that of the base and fibre components is straightforward.
\end{proof}

\section{Future work and open questions}

This section contains the questions related to the above work which I  found interesting, but could not answer for one reason or another. I have neither the time nor expertise to answer all of them myself. I encourage any reader with an answer to contact me.

The symmetric groups have double covers, which are (mostly non-split) group extensions. Taken together, they form a double cover of $\FB$. So this should extend to a double cover of $\FI$. Should such a category exist, it ought to admit a fibration over $\FI$, in which case the present work would inform us of its representation stability properties, thus giving stability results for families of representations of the double covers of the symmetric groups.

We should be able to use a pointwise formula for Kan extensions of pseudofunctors to define an induction-restriction adjunction for indexed categories \cite{PsKanExt}. This adjunction would then transfer to extensions of $\FI$-type categories. 

A question suggested to me by Wee Liang Gan is if you start with an indexed category over a locally noetherian category (i.e.\ all finitely generated modules over a noetherian ring are noetherian), does the Grothendieck construction produce another locally noetherian category?

Nir Gadish pointed out to me that one of the main motivations to think about categories of $\FI$-type was to generalize character polynomials, and it would be natural to explore how the ring of character polynomials changes under the Grothendieck construction. For example, in the case of $\FI$, the ring of character polynomials is the polynomial ring in infinitely many variables. Then how would that change when passing from $\FI$ to $\FI_G$ and to non-split extensions?

In \cite{Grobner}, Sam and Snowden define a property for functors called ``property (F)'', and show that if $\Phi \maps C \to C'$ has this property, and $\Rep_k(C)$ is Noetherian, then $\Rep_k(C')$ is also Noetherian. A na\"ive attempt to use this in the present context would be to investigate when a fibration has property (F). However, note that a fibration is generally going to map from a category we know less about to one we know more about, so this is the wrong direction. Instead, it is more interesting to investigate when a fibration with a $\FI$-type base category admits a property (F) section.

\bibliographystyle{alpha}
\bibliography{references}

\begin{thebibliography}{BFMP20}

\bibitem[BFMP20]{NetworkModels}
John~C. Baez, John {Foley}, Joe Moeller, and Blake~S. {Pollard}.
\newblock Network models.
\newblock {\em Theory and Applications of Categories}, 35(20):700--744, 2020.
\newblock Available at
  \href{http://www.tac.mta.ca/tac/volumes/35/20/35-20abs.html}{http://www.tac.mta.ca/tac/volumes/35/20/35-20abs.html}.

\bibitem[Bor94]{Handbook2}
Francis Borceux.
\newblock {\em Handbook of Categorical Algebra. 2}, volume~51 of {\em
  Encyclopedia of Mathematics and its Applications}.
\newblock Cambridge University Press, 1994.

\bibitem[BS07]{BaezShulman}
John~C. Baez and Michael Shulman.
\newblock Lectures on n-categories and cohomology.
\newblock Available as \href{http://arxiv.org/abs/math/0608420}{arXiv:0608420},
  2007.

\bibitem[CEFN14]{FImodNoetherian}
Thomas Church, Jordan~S. Ellenberg, Benson Farb, and Rohit Nagpal.
\newblock {FI}-modules over {Noetherian} rings.
\newblock {\em Geometry \& Topology}, 18(5):2951--2984, 2014.
\newblock arXiv: 1210.1854.

\bibitem[CZ21]{Devissage}
Jonathan~A. Campbell and Inna Zakharevich.
\newblock Devissage and localization for the {Grothendieck} spectrum of
  varieties.
\newblock Available as \href{
  https://arxiv.org/abs/1811.08014v3}{arXiv:1811.08014}, 2021.

\bibitem[Gad17]{FItypeCats}
Nir Gadish.
\newblock Categories of {FI} type: a unified approach to generalizing
  representation stability and character polynomials.
\newblock {\em Journal of Algebra}, 480:450--486, 2017.
\newblock Available as
  \href{https://arxiv.org/abs/1608.02664}{arXiv:1608.02664}.

\bibitem[GL15]{EICat}
Wee~Liang Gan and Liping Li.
\newblock Noetherian property of infinite {EI} categories.
\newblock {\em New York Journal of Mathematics}, 21:369--382, 2015.

\bibitem[Gra66]{Grayfibredandcofibred}
John~W. Gray.
\newblock Fibred and cofibred categories.
\newblock In {\em Proc. Conf. Categorical Algebra (La Jolla, Calif., 1965)},
  pages 21--83. Springer, New York, 1966.

\bibitem[Gro61]{Grothendieckcategoriesfibrees}
Alexander Grothendieck.
\newblock Cat{\'e}gories fibr{\'e}es et descente, 1961.
\newblock Seminaire de g{\'e}ometrie alg{\'e}brique de l'Institut des Hautes
  {\'E}tudes Scientifiques (SGA 1), Paris.

\bibitem[Her94]{FibredAdjunctions}
Claudio Hermida.
\newblock On fibred adjunctions and completeness for fibred categories.
\newblock In {\em Recent trends in data type specification ({C}aldes de
  {M}alavella, 1992)}, volume 785 of {\em Lecture Notes in Comput. Sci.}, pages
  235--251. Springer, Berlin, 1994.

\bibitem[Jac99]{Jacobs}
Bart Jacobs.
\newblock {\em Categorical logic and type theory}, volume 141 of {\em Studies
  in Logic and the Foundations of Mathematics}.
\newblock North-Holland Publishing Co., Amsterdam, 1999.

\bibitem[Joh02]{Elephant1}
Peter~T. Johnstone.
\newblock {\em Sketches of an {E}lephant: {A} {T}opos {T}heory {C}ompendium.
  {V}ol. 1}.
\newblock Oxford Logic Guides. The Clarendon Press Oxford University Press, New
  York, 2002.

\bibitem[JY21]{2DCats}
N.~Johnson and D.~Yau.
\newblock {\em 2-Dimensional Categories}.
\newblock Oxford U.\ Press, 2021.
\newblock Available as
  \href{https://arxiv.org/abs/2002.06055}{arXiv:2002.06055}.

\bibitem[Lum17]{279985}
Peter~LeFanu Lumsdaine.
\newblock Can we always make a strictly functorial choice of
  pullbacks/re-indexing?
\newblock MathOverflow, 2017.
\newblock
  \href{https://mathoverflow.net/q/279985}{https://mathoverflow.net/q/279985}.

\bibitem[Man21]{monoidGrothConst}
Graham Manuell.
\newblock Monoid extensions and the {Grothendieck} construction.
\newblock Available as \href{
  https://arxiv.org/abs/2112.10288}{arXiv:2112.10288}, 2021.

\bibitem[Nun18]{PsKanExt}
F.~L. Nunes.
\newblock Pseudo-kan extensions and descent theory.
\newblock {\em Theory and Applications of Categories}, 33(15):390--444, 2018.

\bibitem[RSW00]{TwistedAction}
Iain Raeburn, Aidan Sims, and Dana~P. Williams.
\newblock Twisted actions and obstructions in group cohomology.
\newblock In Joachim Cuntz and Siegfried Echterhoff, editors, {\em
  C*-{Algebras}}, pages 161--181. Springer Berlin Heidelberg, Berlin,
  Heidelberg, 2000.

\bibitem[Shu08]{FramedBicats}
Michael Shulman.
\newblock Framed bicategories and monoidal fibrations.
\newblock {\em Theory Appl. Categ.}, 20(18):650--738, 2008.

\bibitem[SS17]{Grobner}
Steven~V. Sam and Andrew Snowden.
\newblock Gr\"obner methods for representations of combinatorial categories.
\newblock {\em J. Amer. Math. Soc.}, 30(1):159--203, 2017.

\bibitem[SS19]{RepGmaps}
Steven~V. Sam and Andrew Snowden.
\newblock Representations of categories of {G}-maps.
\newblock {\em Journal für die reine und angewandte Mathematik (Crelles
  Journal)}, 2019(750):197--226, 2019.

\bibitem[Str20]{alaBenabou}
Thomas Streicher.
\newblock Fibred categories \`a la {Jean B\'enabou}.
\newblock Available as
  \href{https://arxiv.org/abs/1801.02927}{arXiv:1801.02927}, 2020.

\bibitem[Vis05]{Vistoli}
Angelo Vistoli.
\newblock Grothendieck topologies, fibered categories and descent theory.
\newblock In {\em Fundamental Algebraic Geometry}, volume 123 of {\em
  Mathematical Surveys and Monographs}, pages 1--104. American Mathematical
  Society, 2005.
\newblock Available as
  \href{https://arxiv.org/abs/math/0412512v4}{arXiv:0412512}.

\end{thebibliography}

\end{document}